\documentclass[letterpaper, 10 pt, conference]{ieeeconf}  

\IEEEoverridecommandlockouts                              

\usepackage{color}
\usepackage{amsmath}
\usepackage{tikz}
\usepackage{float}
\usepackage{algorithm}
\usepackage{algorithmic}
\usepackage{balance}

\newtheorem{theorem}{Theorem}
\newtheorem{corollary}{Corollary}
\newtheorem{definition}{Definition}
\newtheorem{lemma}{Lemma}

\newtheorem{prop}{Proposition}

\DeclareMathOperator*{\argmax}{argmax}
\DeclareMathOperator*{\argmin}{argmin}

\newcommand{\E}{\mathcal{E}}

\newcommand{\comment}[1]{}

\newcommand{\tr}[1]{\ensuremath{\textbf{tr}\left(#1\right)}}

\newcommand{\Ne}[1]{\ensuremath{{\mathcal N_#1}}}

\newcommand{\tp}{\ensuremath{^{\mathsf{T}}}}

\newcommand{\Econ}{\ensuremath{E_{con}}}
\newcommand{\Ec}{\ensuremath{E^c}}

\title{\LARGE \bf
Optimizing the Coherence of Composite Networks
}

\author{Erika Mackin and Stacy Patterson
    \thanks{This work was supported in part by NSF Grants CNS-1527287 and CNS-1553340.} 
\thanks{ Erika Mackin and Stacy Patterson are with Department of Computer Science, Rensselaer Polytechnic Institute, Troy, New York 12180, USA. Email:
    {\tt\small mackie2@rpi.edu, sep@cs.rpi.edu} } }

\begin{document}

\maketitle
\thispagestyle{empty}
\pagestyle{empty}

\begin{abstract}
We consider how to connect a set of disjoint networks to optimize the performance of the resulting composite network.
We quantify this performance by the coherence of the composite network, which is defined by an $H_2$ norm of the system.
Two dynamics are considered: noisy consensus dynamics with and without stubborn agents. 
For noisy consensus dynamics without stubborn agents, we derive analytical expressions for the coherence of composite networks in terms of the coherence of the individual networks and the structure of their interconnections.  We also identify  optimal interconnection topologies and give bounds on coherence for general composite graphs. 
 For noisy consensus dynamics with stubborn agents, we  develop  a non-combinatorial algorithm that identifies
connecting edges such that the composite network coherence closely approximates the performance of the optimal composite graph.

\end{abstract}

\section{INTRODUCTION}
Networked systems are becoming ever more important in today's highly connected world. We find such systems in power grids, vehicle networks, sensor networks, and so on.  A problem of particular interest is how to coordinate or synchronize these networks, and in addition, how robust this coordination 
or synchronization is to external disturbances.  With an understanding of the relationship between the network topology and this robustness, it becomes possible to modify a network's topology to optimize performance.

In this paper, we study topology design in networks that take the form of  composite graphs. A composite graph is one  that is formed from a set of disjoint subgraphs and a designed set of edges between them. We analyze these networks under two dynamics: noisy consensus dynamics and noisy consensus dynamics with stubborn agents. In both cases, we investigate how to choose edges to connect subgraphs to optimize the network coherence---a performance measure defined by the $H_2$ norm of the system.
For networks with noisy consensus dynamics and no stubborn agents,  we derive analytical expressions for the coherence of composite networks in terms of the coherence of the individual sub-networks and the structure of their interconnections. We then derive upper and lower bounds for the coherence of general composite graphs. For systems with noisy consensus dynamics with stubborn agents, we prove that coherence is a  submodular function of the edges added to a set of initially disjoint networks, and we use this result to create a greedy algorithm for choosing the connecting edge set for the network. This greedy algorithm yields an edge set that is within a provable bound of the performance of the optimal edge set.
%

Coherence has been used as a measure of network performance in several previous works, for example, \cite{YAOSS11,BJMP12}. \cite{ZSA13}, \cite{SSLD15},
and \cite{F15} describe algorithms and analysis for adding edges to an arbitrary graph to improve its coherence.
Modifying the edge weights within a graph is another approach for optimizing network coherence, which is used in \cite{GBS08, XBK07}. In all of these works, however, the authors consider only edge additions or modifications to a single graph.
Our focus, in contrast, is on how best to connect a set of disjoint subgraphs.
A system of interacting networks is considered in the works \cite{GBHS11,PB12}, where the performance is based on robustness to cascading and random failures.
In \cite{SGPP16}, the authors study the performance of a composite network in terms of the $H_2$ norm of the system, but they consider different dynamics than those presented in this paper. 

The remainder of this paper is structured as follows. Section~\ref{model.sec} describes our system model. Section III  gives analysis and formulas regarding the coherence of composite networks with noisy consensus dynamics. In Section IV, we consider noisy consensus dynamics with stubborn agents and present our greedy algorithm for connecting edge selection. Section V demonstrates the performance of our algorithm through a pair of numerical examples, followed by our conclusion in Section \ref{concl.sec}.

%
%
%


\section{System Model} \label{model.sec}
We consider a graph consisting of a set of $n$ disjoint subgraphs $\{G_1, \ldots, G_n\}$. Each subgraph $G_i = (V_i, E_i)$ is connected and undirected, with $n_i$ nodes.
The objective is to connect these $n$ subgraphs to form a connected \emph{composite graph} $G=(V,E)$, $|V|=N$, where:
\begin{align*}
V=&V_1\cup \ldots \cup V_n \\
E =& E_1 \cup \ldots \cup E_n \cup \Econ,
\end{align*}
where $\Econ$ is a set of undirected edges connecting the subgraphs, i.e.,
 $\Econ \subseteq \{(u,v)~|~ u\in V_i,~v \in V_j,~ j \neq i \}$.
The edge set $\Econ$ is to be selected so as to optimize a desired performance objective. 

The dynamics of each node $j \in V$ is given by
\begin{align}
\dot{x}_{j} = u_{j} + \nu_{j}, \label{nodedyn.eq}
\end{align}
where $u_{j}$ is the control input and $\nu_{j}$ is a zero-mean, unit variance, white stochastic disturbance.
We consider two types of dynamics, described below.  

\subsection{Noisy Consensus Dynamics}
We first consider noisy consensus dynamics, where each node updates its state based on the relative states of its neighbors.
The control input is given by:
\begin{align}
u_j = - \sum_{k \in \Ne{j}}(x_j - x_k). \label{conscontrol.eq}
\end{align}
The dynamics of the the network $G$ can be written as
\begin{align}
\dot{x} &= - L  x + \nu, \label{consdyn.eq}
\end{align}
where $L$ is the Laplacian matrix of the composite graph $G$, i.e.,
\[
L = \left[ \begin{array}{ccc}
L_1 & \hdots & 0 \\
\vdots & \ddots & \vdots \\
0 & \hdots  & L_n 
\end{array} \right] + L_{E_{con}},
\]
where $L_i$ is the Laplacian  matrix of  $G_i$, $i=1 \ldots n$, and $L_{E_{con}}$ is the Laplacian matrix of the graph consisting of all nodes in $V$ and only those edges in $E_{con}$.

We quantify the performance of the network by the \emph{network coherence}, which is defined as follows:
\[
H_C(G):=  \lim_{t \rightarrow \infty}  \sum_{j=1}^{N} \textbf{var}\left(x_j - \frac{1}{N}\sum_{k=1}^{N} x_k \right).
\]
The network coherence is the total steady-state variance of the deviations from the average of the current node states.
It has been shown that~\cite{BJMP12,YSL10}:
\[
H_C(G) = \frac{1}{2} \tr{L^{\dagger}},
\]
where  $L^{\dagger}$ is the pseudo-inverse of $L$.


\vspace{-.2cm}
\subsection{Stubborn Agent Dynamics}
We also consider noisy consensus dynamics with \emph{stubborn agents}.
The nodes execute a consensus law, each with some degree of stubbornness, as defined by the scalar  $d_j \geq 0$.
We assume that, in each subgraph, at least one $d_j$ is strictly greater than 0.  
The control input is given by:
\begin{align}
u_j = - \sum_{k \in \Ne{j}} (x_{j} - x_{k}) - d_{j} x_{j}. \label{nodecontrol.eq}
\end{align}

The dynamics of the composite network can be written as:
\begin{align}
\dot{x} &= - Q + \nu, \label{stubdyn.eq}
\end{align}
where
\[
Q = \left[ \begin{array}{ccc}
Q_1 & \hdots & 0 \\
\vdots & \ddots & \vdots \\
0 & \hdots  & Q_n 
\end{array} \right]+L_{E_{con}},
\]
Here, $Q_i = L_i + D_i$, where $D_i$ is the diagonal matrix of degrees of stubborness for graph $G_i$, and $L_{E_{con}}$ is as defined for the noisy consensus dynamics. 


We again quantify the performance of a graph $G$ by an $H_2$ norm,  
\[
H_S(G) = \tr{ \int_0^{\infty} e^{-2Qt}dt} = \frac{1}{2} \tr{ Q^{-1}}.
\]
Note that  if $G$ is connected and at least one $d_j > 0$, then $Q$ is positive definite~\cite{RJME09}.

The dynamics in (\ref{stubdyn.eq}) are a variation of the dynamics for noise-corrupted leaders presented in~\cite{LMJ14}, where each node $j$ with $d_j>0$ plays the role of a leader.  In our system, we allow that any number of agents may be leaders, including every node in the network.
These dynamics can also be given a different interpretation as leader-follower consensus dynamics with noise-free leaders, similar to those presented in~\cite{PB10,LMJ14,CBP14}.  
Let $G'$ be the graph formed from $G$ by adding a single node $s$  and creating an edge from
 each node $j$ in $V$ to $s$ with edge weight $d_j$. All other edge weights are equal to 1.
Let node $s$ be the single leader node, with noise-free dynamics, i.e. $\dot{x}_{s} = 0$,
and let all other nodes be follower nodes, governed by the dynamics in (\ref{nodedyn.eq}) with the control input in (\ref{nodecontrol.eq}).
Let $L'$ be the weighted Laplacian matrix of $G'$, and let $L'_f$ be the sub-matrix of $L'$ with the row and column corresponding to node $s$ removed.  It has been shown that the total steady-state variance of the deviation from the leader's state is given by~\cite{BH08}:
\[
H_f(G) = \frac{1}{2} \tr{{L'}_f^{-1}}.
\]
Observing that $L' = Q$, it holds that $H_S(G) = H_f(G)$.

\section{Composite Graphs with Noisy Consensus Dynamics}
We first consider systems with noisy consensus dynamics and no stubborn agents.  For such networks, we refine the definition
of a composite network in the following way. 
For each of the $n$ disjoint subgraphs $G_i$,
a single node $l_i \in V_i$ is used to connect $G_i$ to the other subgraphs.  
We call these nodes \textit{bridge nodes}. An example of a bridge node connecting two subgraphs is shown in in Figure \ref{bnodes.fig}. Each composite network $G$ of $n$ subgraphs will accordingly have $n$ bridge nodes, where each $l_i$ is connected to at least one other bridge node $l_j$.

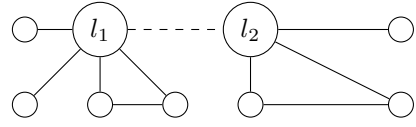
\begin{figure}
\centering
\begin{tikzpicture}

\node[shape=circle,draw=black] (1) at (0,0) {};
\node[shape=circle,draw=black] (2) at (1,1) {$l_1$};
\node[shape=circle,draw=black] (3) at (2,0) {};
\node[shape=circle,draw=black] (8) at (0,1) {};
\node[shape=circle,draw=black] (9) at (1,0) {};

\node[shape=circle,draw=black] (5) at (3,0) {};
\node[shape=circle,draw=black] (4) at (3,1) {$l_2$};
\node[shape=circle,draw=black] (6) at (5,0) {};
\node[shape=circle,draw=black] (7) at (5,1) {};

\foreach \from/\to in {1/2,2/3,4/5,4/6,4/7,5/6,2/8,9/3,9/2}
    \draw (\from) -- (\to);
\foreach \from/\to in {2/4}
    \draw (\from) [dashed] -- (\to);
\end{tikzpicture}
\caption{Two bridge nodes, $l_1$ and $l_2$, in a composite graph $G$ form the backbone graph.}
\label{bnodes.fig}
\end{figure}

The \emph{backbone graph} is the graph defined by the bridge nodes and the edges between them, $B=(V_B, E_B)$, $V_B=\{l_1, \ldots, l_n\}$, $E_B=\{(l_i, l_j) ~|~l_i, l_j, \in V_B\}$.  
The edge set $E_B$ corresponds to $\Econ$ in the composite graph.

 Our goal is to analyze the coherence of the composite graph in terms of its subgraphs, bridge nodes, and backbone graph topology. To do this, we exploit the connection between coherence and effective resistance in electrical networks.

\subsection{Resistance in Electrical Networks}

Consider a connected graph $G=(V,E)$ with $N$ nodes that represents an electrical network, where each edge is a unit resistor. The \emph{resistance distance} $r(u,v)$
between nodes $u$ and $v$ is the potential distance between $u$ and $v$ when a 1-A current source runs between them~\cite{KR93}. 
The \emph{effective resistance} of $G$ is the sum of the resistance distances between each pair of nodes~\cite{KR93}:
\begin{equation}
\Omega_G  = \frac{1}{2}\sum_{u,v\in V_G}r(u,v) = \sum_{u<v\in V_G}r(u,v).
\end{equation}

Coherence is related to effective resistance as follows \cite{PB14}:
\begin{align}
\label{resToCoh}
H_C(G)=\frac{\Omega_G}{2N}.
\end{align}


We use the following lemmas in our analysis.
\begin{lemma}[\cite{KR93}]
\label{cutpoint}
For any graph $G=(V,E)$, let $A=(V_A, E_A)$ and $B=(V_B,E_B)$ be two subgraphs such that $V_A \cup V_B = V$, $V_A \cap V_B = \{x\}$, $E_A \cup E_B = E$, and $E_A \cap E_B = \emptyset$. In other words $G$, is partitioned into two components $A$ and $B$ that share only a single vertex $\{x\}$.
The resistance distance between any two vertices $u$, $v$ with $u \in V_A$ and $v \in V_B$ is: 
\begin{equation}
r(u,v) = r(u,x) + r(x,v).
\end{equation}
\end{lemma}

\begin{lemma}[\cite{KR93}]
\label{graphDist}
For all vertex pairs $u,v \in G$, the graph distance $d(u,v)$ is such that $d(u,v) \geq r(u,v)$, with equality if and only if there is exactly one path between $u$ and $v$.
\end{lemma}
In the case of tree graphs, we note that $d(u,v)=r(u,v)$.

\subsection{Coherence in General Composite Graphs}
We now analyze coherence for a general composite graph with an arbitrary backbone graph topology. To do so, we make use of the following definition.
\begin{definition}
Consider graph $G = (V, E)$. The \emph{resistance centrality} of a node $v \in V$ is
\begin{align}
C(v)=\sum_{u \in V \atop u \neq v}r(u,v).
\end{align}
\end{definition}
We observe that resistance centrality is inversely proportional to the information centrality measure defined in \cite{SZ89}.
We use  $C_i(v_i)$ to denote the resistance centrality of a node $v_i$ in subgraph $G_i$, computed only over the subgraph $G_i$.
Using this definition, we derive a formula for $H_C(G)$  in terms of the the coherence of the subgraphs, the choice of bridge nodes, and the topology of $B$.
 \begin{theorem}
 \label{coh.thm}
Consider a composite graph $G$ with backbone graph $B$ ($|V_B|=n$):
\begin{enumerate}
\item The coherence of $G$ is:
\begin{align}
&H_C(G)= \frac{1}{2N} \Big(\sum_{i=1}^n 2 n_i H_C(G_i) \nonumber \\
&+ \sum_{i=1}^n\sum_{j=i+1}^n|V_i||V_j|r(l_i,l_j) + \sum_{i=1}^n |V - V_i| C_i(l_i) \Big).  \label{cohEq}
\end{align}
\item To minimize the coherence of $G$, $B$ should be defined such that $l_i = \arg \min_{v \in V_i} C_i(v)$.
\end{enumerate}
\end{theorem}

To prove this theorem, we  use the following proposition, which immediately follows from Lemma \ref{cutpoint}.
\begin{prop}
\label{effres.prop} The resistance distance between two nodes $u_i,~v_j \in V$ where $u_i \in V_i$, $v_j \in V_j$, $i \neq j$ is $r(u_i,v_j) = r(u_i,l_i) + r(l_i,l_j)+r(l_j,v_j)$. 
\end{prop}
We now prove the theorem.

\begin{proof}
We find the effective resistance of $G$ and then use this to find $H_C(G)$.
Let $G = (V,E)$  be a composite graph, i.e., $V = V_1 \cup \ldots \cup V_n$ and $E = E_1 \cup \ldots \cup E_n \cup E_{B}$.
$G$ is constructed such that every pair of subgraphs $G_i$, $G_j$, $i \neq j$ is connected only through their respective bridge nodes $l_i$ and $l_j$. 
By applying Lemma~\ref{cutpoint}, we can define $\Omega_G$ in terms of  $\Omega_i$, the effective resistance of $G_i$, for $i = 1 \ldots n$,
and the resistance distances of  each edge $e \in E_B$: 
 \begin{align}
 \Omega_G &=  \frac{1}{2}\sum_{i=1}^n \sum_{j=1}^n \left(\sum_{u \in V_i} \sum_{v \in V_j} r(u,v) \right)  \label{Rg1}\\
 &=\sum_{i=1}^n\Omega_i + \sum_{i=1}^n \sum_{j>i}^n \left(\sum_{u \in V_i} \sum_{v \in V_j} r(u,v) \right). \label{Rg2}
\end{align}
To obtain (\ref{Rg2}) from (\ref{Rg1}), note that when $i=j$, $\sum_{u \in V_i} \sum_{v \in V_j} r(u,v) = \Omega_i$.

For $i \neq j$, each term $\sum_{u\in V_i}\sum_{v\in V_j}r(u,v)$ in (\ref{Rg2}) can be rewritten as 
\begin{align}
\sum_{u\in V_i}\sum_{v\in V_j}r(u,l_i)+r(l_i,l_j)+r(l_j,v), \label{dblsum}
\end{align}
 as noted in Proposition \ref{effres.prop}.
In turn, the double sum  in (\ref{dblsum}) can be simplified to $|V_j|C_i(l_i) + |V_i||V_j|r(l_i,l_j) + |V_i|C_j(l_j)$.
The formula for $\Omega_G$ now becomes:\\
\begin{align*}
\Omega_G =& \sum_{i=1}^n\Omega_i \\
+ &\sum_{i=1}^n\sum_{j=i+1}^n\left( |V_j|C_i(l_i) + |V_i||V_j|r(l_i,l_j) + |V_i|C_j(l_j) \right) \\
=& \sum_{i=1}^n\Omega_i + \sum_{i=1}^n\sum_{j=i+1}^n\left( |V_j|C_i(l_i) + |V_i|C_j(l_j) \right)  \\
&+ \sum_{i=1}^n\sum_{j=i+1}^n|V_i||V_j|r(l_i,l_j).
\end{align*}
We can see that $\sum_{i=1}^n\sum_{j=i+1}^n\left( |V_j|C_i(l_i) + |V_i|C_j(l_j) \right)$
is $C_i(l_i)$ multiplied by the total number of vertices in $G-G_i$, therefore $\Omega_G$ is:
\begin{equation}
\label{resEq}
\Omega_G =  \sum_{i=1}^n\Omega_i +  \sum_{i=1}^n\sum_{j=i+1}^n|V_i||V_j|r(l_i,l_j) + \sum_{i=1}^n |V - V_i| C_i(l_i).
\end{equation}
Then, we divide use (\ref{resToCoh}) to replace $\Omega_i$ with $2n_iH_C(G_i)$ and divide (\ref{resEq}) by $2N$ to finally obtain (\ref{cohEq}). 

All terms in $H_C(G)$ are constant for any fixed set of graphs $G_i$ and backbone graph $B$, except for $C_i(l_i)$. Thus, we minimize $H_C(G)$ by selecting each bridge node $l_i$ to be the vertex with the minimum resistance centrality in $G_i$.
\end{proof}
\subsection{Analysis of Coherence in  Backbone Graph Structures}

We now explore specific backbone graph topologies. In addition to deriving formulae for  coherence, we are interested in identifying backbone graphs that
minimize the coherence of the composite graph.

First, consider a composite graph with a tree backbone graph with $|V_B|=n$. The coherence of such a composite graph can be derived from Theorem \ref{coh.thm} by using Lemma \ref{graphDist} to replace $r(l_i,l_j)$ with the graph distance $d(l_i,l_j)$ between bridge nodes to obtain:
\begin{align}
\label{tree}
H_C(G) =&    \frac{1}{2N} \Big( \sum_{i=1}^n 2n_i H_C(G_i) + \sum_{i=1}^n\sum_{j=i+1}^n d(l_i,l_j) |V_i||V_j| \nonumber \\
&+ \sum_{i=1}^n |V - V_i| C_i(l_i) \Big).
\end{align}

With this expression we can show that the star backbone graph is the optimal tree backbone graph topology.
\begin{corollary}
\label{treec}
The optimal composite graph with a tree backbone graph, $B=(V_B,E_B)$, $|V_B|=n$, has a star graph for $B$,
and the bridge node $l_c$ in the center of the star is such that $l_c \in V_c$, where $V_c \in \argmax \limits_{i} |V_i|$. \end{corollary}
\begin{proof}
In a star graph, the resistance distance between the bridge node of the central graph $G_c$ and all other subgraphs $G_i$, $c \neq i$, is 1, and the resistance distance between the bridge nodes of all other subgraphs $G_i$ and $G_j$, $i,j \neq c$, is 2. 

When calculating $H_C(G)$, $d(l_i,l_j) |V_i||V_j|$ is computed for all combinations of $i$ and $j \leq n$. Therefore, to minimize $\sum_{i=1}^n\sum_{j=i+1}^n d(l_i,l_j) |V_i||V_j|$, and thus also minimize (\ref{tree}), we choose a subgraph $V_c \in \argmax \limits_{i} |V_i|$ to be the center of the star graph.
Then,  $\sum_{i=1}^n\sum_{j=i+1}^n d(l_i,l_j) |V_i||V_j|$ becomes  $\sum_{i \neq c} |V_c| |V_i| + \sum_{i \neq c} \sum_{j=i+1, j \neq c} 2\cdot |V_i||V_j|$.
No other arrangement of subgraphs in a tree can reduce $H_C(G)$ more.
\end{proof}

Now consider a composite graph with a line backbone graph of size $n$. The coherence of such a composite graph is again derived from (\ref{cohEq}) and (\ref{tree}), but here $d(l_i,l_j)=j-i$ for all $i>j$. Therefore:
\begin{align}
&H_C(G) =  \frac{1}{2N} \Big( \sum_{i=1}^n  2n_i H_C(G_i) \nonumber \\
&+ \sum_{i=1}^n\sum_{j=i+1}^n (j-i)|V_i||V_j| \sum_{i=1}^n |V - V_i|C_i(l_i) \Big). \label{cohLine}
\end{align}

\begin{corollary}
The optimal composite graph $G$ with a line backbone graph $B=(V_B, E_B)$  where the nodes of $B$, $V_B=\{l_{s_1}, l_{s_2}, \ldots,l_{s_n}\}$, are ordered from left to right in the path graph and the subgraphs $G_1,G_2, \ldots, G_n$ are ordered by decreasing vertex set size. Then if $c={\lfloor \frac{n}{2}\rfloor}$, we can assign the vertices of $V_B$ as follows: $l_{s_c}=l_1, ~ l_{s_{c+1}}=l_2, ~ l_{s_{c-1}}=l_3, l_{s_{c+1}}=l_4, ~ l_{s_{c-2}}=l_5, ~\text{etc.}$, where $l_i$ is the bridge node of $G_i$. 
\end{corollary}
\begin{proof}
As in Corollary \ref{treec}, to optimize (\ref{cohLine}) we need only to optimize $\sum_{i=1}^n\sum_{j=i+1}^n (j-i)|V_i||V_j|$. Therefore we need to find the optimal arrangement of subgraphs in order to minimize this sum, which can be done by finding the ordering such that as $|V_i||V_j|$ increases, $j-i$ decreases. Clearly, this is done by placing the subgraphs along the line backbone in a way that minimizes the distance of the largest subgraphs to all other subgraphs and maximizes the distance of the smallest subgraphs to all other subgraphs. This requirement is fulfilled by placing the largest subgraph at the center of the line, placing the smallest subgraphs on the endpoints of $B$, and arranging the subgraphs in between closer to the center or to the ends according to their size.

\end{proof}

The coherence of a composite graph with a ring backbone graph $|V_B|=n$ is derived from (\ref{cohEq}) by noting that in a ring graph with $n$ nodes, $r(l_i,l_j) = \frac{(j-i)(n-(j-i))}{n}$. We then get the following formula:
\begin{align*}
H_C(G) &= \frac{1}{2N} \Big( \sum_{i=1}^n 2 n_i H_C(G_i)  + \sum_{i=1}^n |V - V_i| C_i(l_i)\\
&+   \sum_{i=1}^n\sum_{j=i+1}^n \frac{(j-i)(n-(j-i))}{n} |V_i||V_j| \Big).
\end{align*}

Finally we consider the upper and lower bounds for $H_C(G)$ over all subgraph topologies and all backbone graph topologies.
The proof of these results is given in the appendix.


\begin{corollary} \label{bound1.cor}
The lower bound for the coherence of any composite graph with $|V_B|=n$ and $|V_i| = |V_j|=m$ for all $i,j \leq n$ is: 
\begin{equation}
H_C(G) \geq \textstyle \frac{1}{2N} \Big( n(m-1) + 2m^2(n-1) + 2n(n-1)(m-1)\Big). 
\end{equation}
\end{corollary}

Now, we consider the upper bound of $H_C(G)$ for a graph $G$. We first note that with all else held equal, the backbone graph which maximizes $H_C(G)$ is the line graph, since by Lemma \ref{graphDist} only tree backbones have $d(l_i,l_j) \geq r(l_i,l_j)$, and the line graph has the largest diameter of all tree graphs. Since all $|V_i|=m$, the ordering of the subgraphs along the line has no effect on the coherence. We previously derived the formula for the coherence of a line composite graph (\ref{cohLine}).
\begin{corollary} \label{bound2.cor}
The upper bound for the coherence of any composite graph with $|V_B|=n$ and $|V_i| = |V_j|=m$ for all $i,j \leq n$ is: 
\begin{align*}
H_C(G) \leq&  \textstyle \frac{1}{12N}\left(nm(m^2-1)\right) + \frac{1}{12N}\left(nm^2(n^2-1)\right) \\
&~~~~+ \textstyle \frac{1}{4N}\left( nm^2 (m-1)(n-1)\right).
\end{align*}
\end{corollary}

\section{Composite Graphs with Stubborn Agent Dynamics}

We consider the problem of how to select $\Econ$ so as to minimize the coherence $H_S(G)$.
In particular, we assume that only a fixed number of edges $k$
can be chosen.  The optimal edge set can be found by an exhaustive search over all subsets of $k$ edges, however, this approach is computationally intractable for large subgraphs and values of $k$.

Instead, we define a greedy polynomial-time algorithm for selecting the edge set  $E_{con}$.
The pseudocode is given in Algorithm \ref{greedy.alg}.
In each iteration, an edge $e$ is chosen whose addition to $\Econ$ minimizes the coherence of $G$.  This is repeated until $|\Econ|$ has reached the desired size $k$. The input to Algorithm~\ref{greedy.alg}, is the graph $G=(V,\overline{E})$, where $V=V_1 \cup \ldots \cup V_n$ and $\overline{E}=E_1 \cup \ldots \cup E_n$, the matrix $Q$, and  the desired size of $\Econ$, $k$.  

%


\begin{algorithm}[t]
\caption{Greedy algorithm for choosing $\Econ$.} \footnotesize
\begin{algorithmic}
\REQUIRE $G = (V, \overline{E}),Q,k$
\STATE $\Ec \leftarrow \{(u,v)~|~ u\in V_i,~v \in V_j,~ j \neq i \}$
\STATE $|\Econ| \leftarrow \emptyset$
\WHILE{ $ |\Econ | < k $ } 
\STATE $e \leftarrow \{ \argmin \limits_{e \in \E_c} \tr{(Q_{\Econ}+L_e)^{-1}}\}$
\STATE $\Ec \leftarrow \Ec \setminus \{e\}$
\STATE $\Econ \leftarrow \Econ \cup \{e\}$
\ENDWHILE
\end{algorithmic}
\label{greedy.alg}
\end{algorithm}

\section{Numerical Examples}
We illustrate the performance of our greedy algorithm for adding edges to initially disjoint networks. In all examples, we consider stubborn agent dynamics.
The following two numerical examples were produced in Matlab. Both examples were run with two different types of $D$ matrices: $D_I=I$  and $D_R$, where $D_R$ is diagonalized from $d=[d_1, \ldots, d_n]^{\tp}$ where for each $i$, with probability 0.2, we set $d_i = 0$, and with probability $0.8$, we set $d_i$ to a value chosen uniformly at random from $(0,1]$.

\subsection{Adding Edges Between Versus Within Subgraphs}
\begin{figure}
  \includegraphics[scale=.35]{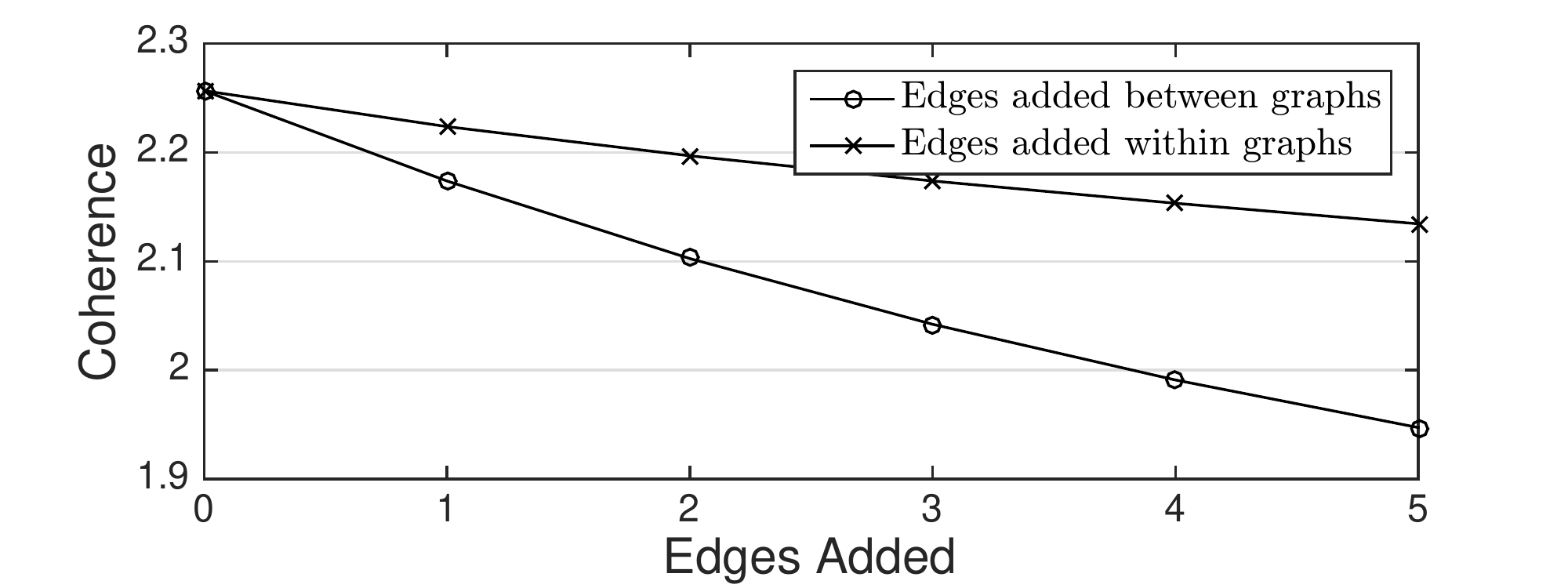}
  \centering
  \caption{Comparison of coherence when adding successive edges between or within the subgraphs of a network.}
  \label{singleEdgesD=I.fig}
\end{figure}

\begin{figure}
  \includegraphics[scale=.35]{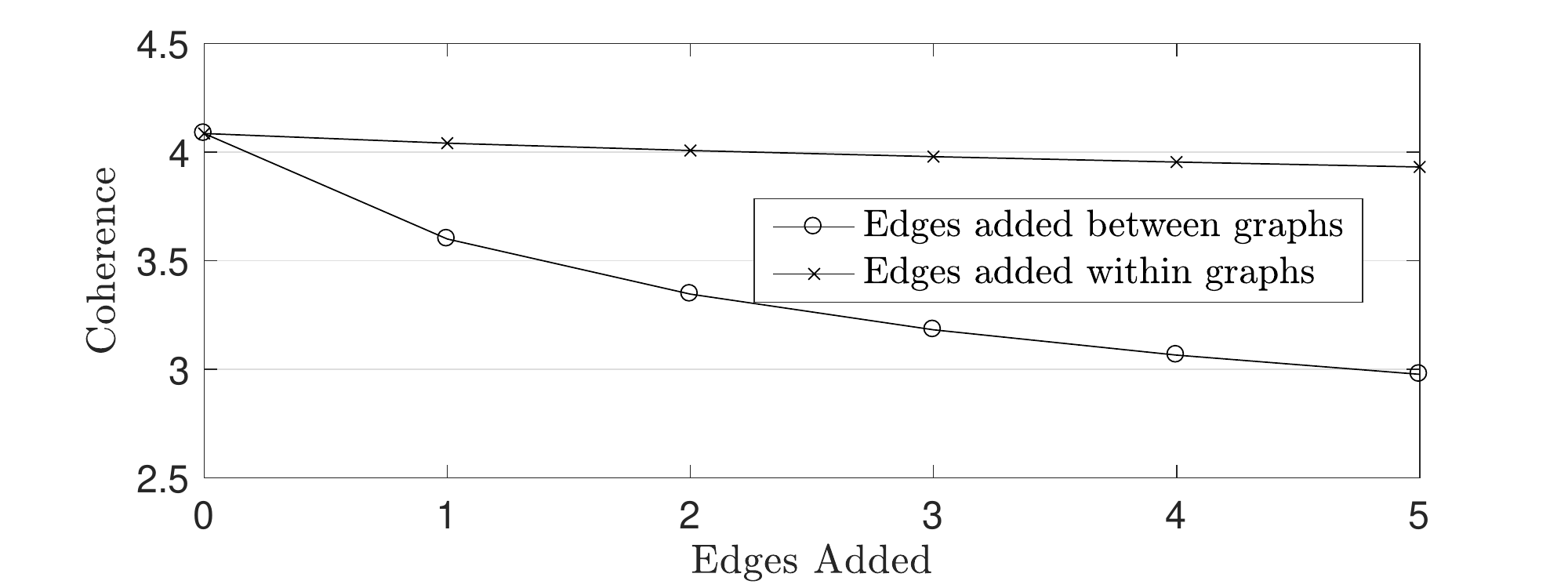}
  \centering
  \caption{Comparison of coherence changes when adding successive edges between or within the subgraphs of a network.}
  \label{singleEdgesD!=I.fig}
\end{figure}

We first demonstrate through numerical examples that the coherence of a network with stubborn agent dynamics will always see more improvement by adding a new edge between the subgraphs rather than within a given subgraph.
 The networks are generated as follows: two disjoint Erd\H{o}s-R\'{e}nyi graphs $G_1$ and $G_2$ with sizes between 8 and 15 are randomly generated to form a network $G=(V_1 \cup V_2, E_1 \cup E_2)=(V,E)$. In Figure \ref{singleEdgesD=I.fig} we use $D_I$ and in Figure \ref{singleEdgesD!=I.fig}, we use $D_R$ for our calculations. For each numerical example, we begin with two candidate edge sets:  the set of edges between the two subgraphs and the set of edges within the two subgraphs. 

We run Algorithm \ref{greedy.alg} to choose the set of edges between the graphs, and then we repeat the algorithm to choose the set of edges within the subgraphs.  We run a series of 20 trials; 
we calculate the resulting coherence of each edge added in every trial and take the average results across all runs.
These average coherence values from adding edges between the subgraphs and within the subgraphs are plotted in the accompanying figures. As we can see in Figure \ref{singleEdgesD=I.fig} and Figure \ref{singleEdgesD!=I.fig}, the coherence of the network with edges added between the subgraphs is always better than the coherence of the network with edges added only within $G_1$ or $G_2$. Note also that initial improvement from adding a single edge is always greater when adding an edge between the subgraphs, rather than adding an edge within the subgraph. 

We have also observed that when Algorithm \ref{greedy.alg} is performed on sets of $n>2$ graphs, so long as more than $n-1$ edges are added, the improvement in coherence from adding a $k^{th}$ edge  up to the $n-1^{th}$ will always be greater when adding an edge between subgraphs than within subgraphs. That is, the best choice will always be to connect another pair of subgraphs, rather than add an edge within an already connected subgraph of $G$.

\subsection{Greedy Versus Optimal $E_{con}$}
\begin{figure}
\centering
  \includegraphics[scale=.35]{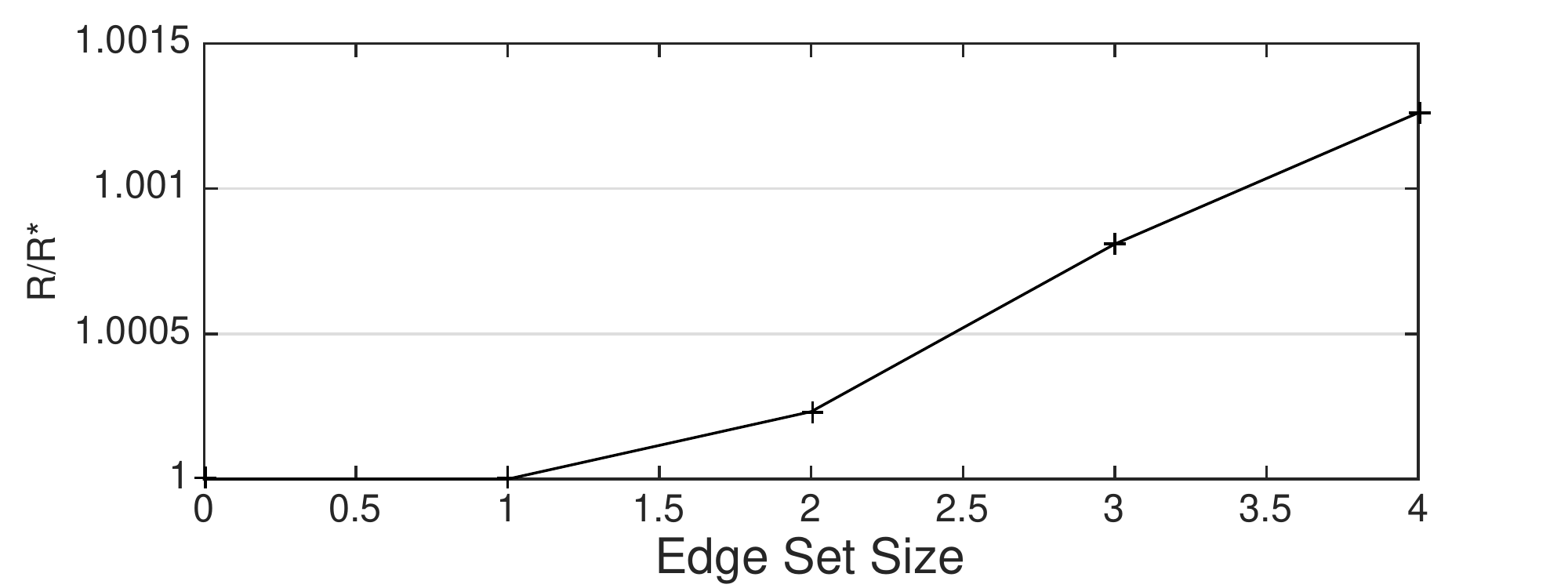}
  \caption{The ratio of the coherence of greedy vs. optimal edge sets for $Q= L + D_I$.}
  \label{greedyRatioD=I.fig}
\end{figure}

\begin{figure}
\centering
  \includegraphics[scale=.35]{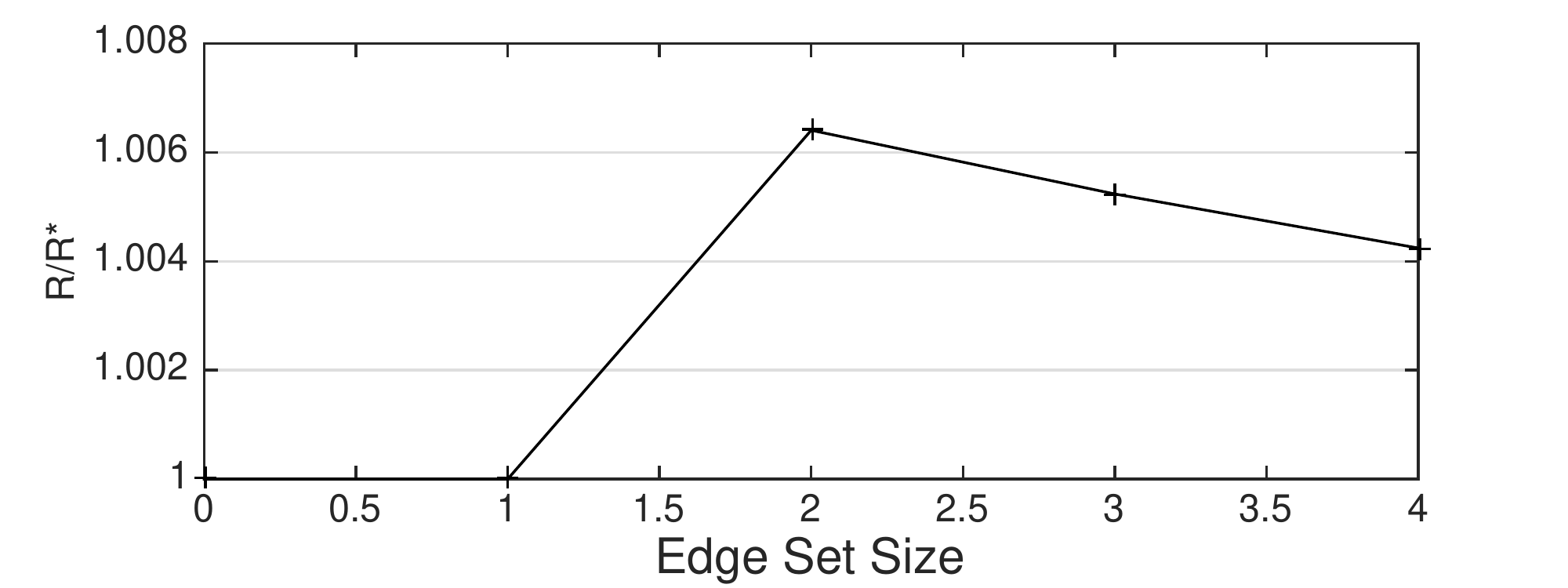}
  \caption{The ratio of the coherence of greedy vs. optimal edge sets for $Q = L + D_R$.}
  \label{greedyRatioD!=I.fig}
\end{figure}
We now demonstrate that the coherence of $G$ composed of two subgraphs with $E_{con}$ of size $k$ chosen by Algorithm \ref{greedy.alg} is very close to the coherence of $G$ with an optimal edge set $E^*$ selected. 
We run simulations with two randomly generated Erd\H{o}s-R\'{e}nyi graphs with size between 4 and 8 (the small graph sizes are due to the combinatorial nature of optimal set selection). As in the previous numerical example, in Figure \ref{greedyRatioD=I.fig} we use $D_I$ and in Figure \ref{greedyRatioD!=I.fig}, we use $D_R$.
For this example, we make no distinction between edges within or between the original subgraphs.

We first use Algorithm \ref{greedy.alg} to choose an edge set of size $k$ to add to the network. Then the optimal set is selected by forming all possible edge sets of size $k$ and calculating the decrease in coherence of adding each of them in turn. The edge set with the greatest decrease is then chosen. The results of both methods are then summed and averaged over a series of 15 trials.

We can see that the performance of the edge sets chosen by Algorithm~\ref{greedy.alg}  are quite close to the performance of the optimal sets, justifying the use of the greedy algorithm.

\section{Example} \label{ex.sec}
%
%

Consider the graphs $G_1$ and $G_2$ shown in Figure \ref{G.fig}. Let their respective bridge nodes be $l_1=2$ and $l_2=4$ and let the edge set of the backbone graph $B$ be $E_B=\{ (2,4)\}$.
We then form the graph $G=\{ \{E_1 \cup E_2\}, \{V_1 \cup V_2 \cup (2,4)\} \}$.

\begin{figure}
\centering
\begin{tikzpicture}
\node[shape=circle,draw=black] (1) at (0,0) {1};
\node[shape=circle,draw=black] (2) at (1,1) {2};
\node[shape=circle,draw=black] (3) at (2,0) {3};

\node[shape=circle,draw=black] (5) at (4,0) {5};
\node[shape=circle,draw=black] (4) at (4,1) {4};
\node[shape=circle,draw=black] (6) at (6,0) {6};
\node[shape=circle,draw=black] (7) at (6,1) {7};

\foreach \from/\to in {1/2,2/3,4/5,4/6,4/7,5/6}
    \draw (\from) -- (\to);
\foreach \from/\to in {2/4}
    \draw (\from) [dashed] -- (\to);
\end{tikzpicture}
\caption{Composite graph $G$.}
\label{G.fig}
\end{figure}
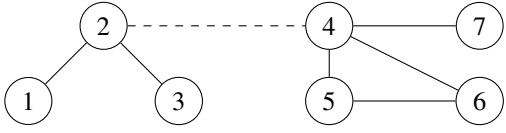

For the given graphs, $R_1 = 4$, $R_2 = 6 \frac{1}{3}$, $C_1(l_1) = 2$, $C_2(l_2) = \frac{7}{3}$. 
We then use (\ref{cohEq}) to calculate the coherence of $G$ to be:  $H_C(G)=\frac{8}{3}$.

We also use the above pair of graphs to illustrate the results of Algorithm \ref{greedy.alg}. In Table \ref{HS.table}, we list the edge sets $E_{con}$ of size $k=1, \ldots, 3$ generated by the algorithm and their corresponding values $H_S(E_{con})$ and compare to the optimal edge set $E^\star$ and $H_S^\star$.
table for $H_S(E_{con})$
\begin{table} 
\scalebox{0.9}{
  \begin{tabular}{ | l || c | r || c | r |}
    \hline
    k & $E_{con}$ & $H_S(E_{con})$ & $E^\star$ & $H_S^\star$\\ \hline \hline
    1 & \{(1,7)\} & 1.6503 & \{(1,7)\} & 1.6503\\ \hline
    2 & \{(1,7), (3,5)\} & 1.4757 & \{(1,5), (3,7)\} & 1.4757\\ \hline
    3 & \{(1,7), (3,5), (1,6)\} & 1.3660 & \{(1,5), (3,6), (2,7)\} & 1.3571 \\
     \hline
  \end{tabular}
  }
  \caption{Edge sets $\Econ$ and $E^\star$ and their corresponding $H_S$ values.}
  \label{HS.table}
\end{table}
 
 \balance 
\section{Conclusion} \label{concl.sec}
We have considered the problem of  how to best connect disjoint subgraphs to optimize the coherence of the composite graph. For systems with noisy consensus dynamics, we have derived several expressions and bounds for the coherence of composite graphs. For systems with stubborn agent dynamics we presented a non-combinatorial algorithm for choosing edges which, when added to the network, closely approximate the performance of the optimal edge set of the same size. Finally, we have demonstrated the performance of this algorithm in numerical examples.

In future work, we plan to investigate analytical expressions for the coherence of composite networks with stubborn agent dynamics, similar to those we derived for composite networks with noisy consensus dynamics.  We also plan to explore the design of composite graphs under additional dynamics and performance measures.

\section*{APPENDIX}
\begin{corollary}
The lower bound for the coherence of any composite graph with $|V_B|=n$ and $|V_i| = |V_j|=m$ for all $i,j \leq n$ is: 
\begin{equation}
H_C(G) \geq \frac{1}{2N} \Big( n(m-1) + 2m^2(n-1) + 2n(n-1)(m-1)\Big).
\end{equation}
\end{corollary}
\begin{proof}
To find the lower bound of $H_C(G)$, we need to find the backbone graph structure that minimizes the coherence of $G$.

The coherence of a composite graph with a complete graph backbone $B$, $|V_B|=n$, is derived from (\ref{cohEq}), using the fact that in a complete graph, $r(l_i,l_j)$ is the same for all $1 \leq i,j \leq n$, and $r(l_i,l_j) = \frac{2n}{n^2}$. The coherence of a complete graph is then:
\begin{align}
\label{complete}
H_C(G) &=  \frac{1}{2N}\Big( \sum_{i=1}^n 2n_iH_C(G_i) +   \frac{2n}{n^2}\sum_{i=1}^n\sum_{j=i+1}^n|V_i||V_j| \nonumber \\
&+ \sum_{i=1}^n |V - V_i| C_i(l_i) \Big).
\end{align}

Clearly the best way to minimize the resistance distance across all edges of the backbone graph is to connect every vertex to every other vertex, resulting in the lowest resistance distance for each edge in the backbone graph. Therefore the composite graph with $|V_B|=n$ and $|V_i|=m$ for all subgraphs $G_i$ will have the highest connectivity when $B$ is a complete graph of size $n$ and each $G_i$ is a complete subgraph of size $m$.

To calculate the effective resistance of this graph $G$, beginning from (\ref{complete}), we substitute $m-1$ for $2n_iH_C(G_i)$, $m$ for $|V_i|$, $|V-V_i|=m(n-1)$, and $(m-1)\frac{2m}{m^2}$ for $C_i(l_i)$ to get: 
\begin{align*}
 &H_C(G)  =  \frac{1}{2N} \Big( \sum_{i=1}^n m-1 +  \frac{2n}{n^2}\sum_{i=1}^n\sum_{j=i+1}^n m^2 \\
&+ \sum_{i=1}^n m(n-1) (m-1)\frac{2m}{m^2} \Big) \\
&= \frac{1}{2N} \Big( n(m-1) + \frac{2n}{n^2}m^2 \frac{n(n-1)}{2} \\
&+ nm(n-1)(m-1)\frac{2m}{m^2} \Big)\\
&= \frac{1}{2N} \Big( n(m-1) + 2m^2(n-1) + 2n(n-1)(m-1)\Big).
\end{align*}
Any other composite graph $G'$ with $|V_i| = m$ and $|V_B|=n$ will therefore have 
\[
H_C(G') \geq \frac{1}{2N} \Big( n(m-1) + 2m^2(n-1) + 2n(n-1)(m-1)\Big).
\]
\end{proof}

\begin{corollary}
The upper bound for the coherence of any composite graph with $|V_B|=n$ and $|V_i| = |V_j|=m$ for all $i,j \leq n$ is: 
\begin{align*}
H_C(G) \leq& \frac{nm(m^2-1) }{12N} + \frac{nm^2(n^2-1)}{12N} \\
&+ \frac{nm^2 (m-1)(n-1)}{4N}.
\end{align*}
\end{corollary}
\begin{proof}
Since the line composite graph has the highest coherence for any graph $G$, in order to maximize $H_C(G)$ we set each $G_i$ to be a line graph of size $m$ and each $l_i \in V_i$ to be an endpoint of the backbone graph $B$. We use the structural properties of a line graph to find that $C_i(l_i)= \sum_{i=1}^{m-1}i=\frac{m(m-1)}{2}$ and $H_C(G_i)= \frac{1}{2n_i}\sum_{i=1}^m \sum_{j=i+1}^m (j-i) = \frac{1}{12n_i}m(m^2-1)$. 
We now substitute $H_C(G_i)=\frac{1}{12 n_i}m(m^2-1)$, $|V-V_i|=m(n-1)$, $|V_i|=m$, and $C_i(l_i)=\frac{m(m-1)}{2}$ into (\ref{cohLine}) to get:
\begin{align*}
H_C(G) &= \\
=& \frac{1}{2N}\Big( \sum_{i=1}^n 2n_i\frac{1}{12n_i}m(m^2-1) + \sum_{i=1}^n \sum_{j=i+1}^n (j-i) m^2 \\
&+ \sum_{i=1}^n m(n-1)\cdot \frac{m(m-1)}{2} \Big). \\
\end{align*}
Simplifying further, we obtain:
\begin{align*}
&H_C(G) =\frac{1}{2N}\Big( \frac{1}{6}m(m^2-1)n + \frac{1}{6}n(n^2-1)m^2 \\
+& nm(n-1)\frac{m(m-1)}{2} \Big). \\
=& \frac{nm(m^2-1) }{12N} + \frac{nm^2(n^2-1)}{12N} + \frac{nm^2 (m-1)(n-1)}{4N}.
\end{align*}
Any other composite graph $G'$ with $|V_i| = m$ and $|V_B|=n$ which does not have both a line backbone graph and line subgraphs $G_i$ will therefore have 
\begin{align*}
H_C(G') \leq& \frac{nm(m^2-1) }{12N} + \frac{nm^2(n^2-1)}{12N} \\
&+ \frac{nm^2 (m-1)(n-1)}{4N}.
\end{align*}
\end{proof}

%







\bibliographystyle{IEEEtran}
\bibliography{grCons}

\end{document}